\documentclass[12pt,a4paper]{article}
\usepackage{t1enc}
\usepackage[utf8]{inputenc}
\usepackage[english]{babel}
\usepackage{amsmath}
\usepackage{amssymb}
\usepackage{amsthm}
\usepackage{hyperref}

\newcommand{\p}{\mathbb{P}}

\newcommand{\F}{\mathbb{F}}

\newcommand{\lra}{\longrightarrow}

\newcommand{\ff}{\mathcal{F}}

\newcommand{\fq}{\mathbb{F}_q}

\newcommand{\kk}{\mathbb K}
\newcommand{\xx}{\mathcal X}
\newcommand{\yy}{\mathcal Y}

\newcommand{\cc}{\mathcal C}

\newcommand{\Char}{\operatorname{char}}
\newcommand{\rank}{\operatorname{rank}}

\newcommand{\Tan}{\operatorname{Tan}}

\newcommand{\fqc}{\overline{\mathbb{F}}_q}

\theoremstyle{plain}
\newtheorem{thm}{Theorem}[section]

\newtheorem{prop}[thm]{Proposition}
\newtheorem{lem}[thm]{Lemma}
\newtheorem{cor}[thm]{Corollary}
\newtheorem{rem}[thm]{Remark}
\newtheorem{ex}[thm]{Example}

\title{Duality for certain multi-Frobenius nonclassical curves in higher dimensional spaces}

\author{Nazar Arakelian}

\begin{document}

\maketitle
%\tableofcontents \newpage
\begin{abstract}
We show how a type of multi-Frobenius nonclassicality of a curve defined over a finite field $\fq$ of characteristic $p$ reflects on the geometry of its strict dual curve. In particular, in such cases we may describe all the possible intersection multiplicities of its strict dual curve with the linear system of hyperplanes. Among other consequence,  using a result by Homma, we are able to construct nonreflexive space curves such that their tangent surfaces are nonreflexive as well, and the image of a generic point by a Frobenius map is in its osculating hyperplane. We also obtain generalizations and improvements of some known results of the literature.
\end{abstract}

\emph{Keywords}: Algebraic curves, dual curves, finite fields.

\section {Introduction}\label{intro}

For an algebraically closed field $\kk$, let $\xx \subset \p^n(\kk)$ be a (projective, geometrically irreducible, algebraic) curve defined over $\kk$, where $\p^n(\kk)$ denotes the $n$-dimensional projective space defined over $\kk$. It is well known from the literature that, in contrast to the zero characteristic case, we can expect many pathological behavior of the geometry of $\xx$ when the characteristic of $\kk$ is positive. For instance, we know that there are $n+1$ non-negative distinct integers $\varepsilon_0,\ldots,\varepsilon_n$ that are realized as intersection multiplicities of $\xx$ and hyperplanes of $\p^n(\kk)$ at a generic point of $\xx$. The sequence of such numbers is called the order sequence of $\xx$. When $\Char(\kk)=0$, these numbers are precisely $0,1,2,\ldots,n$. However, when $\Char(\kk)>0$ this is not necessarily true, and the determination of such sequence can be a hard task in many cases. For details, see e.g. \cite{SV}.

Another remarkable example of differences between the Theory of Curves in zero and positive characteristic lies on the reflexivity and duality problems. If $(\p^{n}(\kk))^{'}$ denotes the dual projective space, then the strict dual of $\xx$, denoted by $\xx^{'}$, is defined as the Zariski closure in $(\p^{n}(\kk))^{'}$ of the set of osculating hyperplanes at the nonsingular locus of $\xx$. The strict dual curve $\xx^{'}$ is a subcover of $\xx$, and the cover $\gamma:\xx \lra \xx^{'}$ is called the strict Gauss map of $\xx$. Denote by $\gamma^{'}$ the strict Gauss map of $\xx^{'}$ and by $\xx^{''}$ the strict dual of $\xx^{'}$. If $\Char(\kk)=0$, one has $\xx^{''}=\xx$ and $\gamma^{'}\circ \gamma$ is the identity map. This situation does not necessarily hold in positive characteristic. In fact, if $\Char(\kk)=p>0$ and the order sequence of $\xx$ is not the classical sequence $(0,1,\ldots,n)$, then both separable and inseparable degrees of $\gamma$ can be positive, see \cite{Ar,GV0} for instance. Moreover, it follows from a result by Kaji in \cite{Ka2} that given $(\cc,r,s)$, where $\cc$ is a plane curve defined over $\kk$ and $r,s$ are positive integers with $\gcd(p,r)=1$, there exists a plane curve $\xx$ such that $\xx^{'}=\cc$, and the strict Gauss map of $\xx$ has separable and inseparable degree equal to $r$ and $p^s$,  respectively.

There are many other examples of pathological phenomena that happens to curves defined over fields of positive characteristic, see for instance \cite{Fu,He,HV,Ho,Ho2,Ka,Kl}. In this paper, we study how two of these phenomena interact to each other for curves defined over a finite field $\fq$, that is, we investigate how certain type of multi-Frobenius nonclassicality of a curve defined over $\fq$ can affect the geometry of its strict dual, and vice versa. More precisely, let $\xx \subset \p^n(\fqc)$ be a curve defined over $\fq$, where $q$ is a power of a prime $p$. The type of $\F_{q^m}$-Frobenius nonclassicality that we are considering here is the one that the image of a generic point of $\xx$ by the $\F_{q^m}$-Frobenius map belongs to the osculating hyperplane to $\xx$ at such point, where $m>0$. After presenting the background results in Section 2, in Section 3, we provide a characterization of curves that are $\F_{q^m}$-Frobenius nonclassical for $n-1$ different powers of $q$ (in the sense that we just describe) under certain conditions. This characterization will be in terms of properties of its strict duals, namely the order sequence, the strict Gauss map and the Frobenius nonclassicality as well. We will also provide less technical sufficient conditions for the validity of an implication of this characterization. In particular, we are able to describe the order sequence and $\F_{q^m}$-Frobenius order sequence of the strict duals of such multi-Frobenius nonclassical curves for different values of $m$. As well as the order sequence of a curve defined over $\fq$, the determination of $\F_{q^m}$-Frobenius order sequence can be difficult in many instances. However, the Frobenius order sequence of a curve is a key ingredient to obtain bounds for the number of its $\F_{q^m}$-rational points, see \cite{AB,SV}.   

Finally, as an application of our results, in Section 4 we construct nonreflexive curves in projective spaces of dimension $n \geq 3$ that are multi-Frobenius (resp. Frobenius) nonclassical when $n>3$ (resp. $n=3$). Thanks to a Homma's result, when $n=3$, this construction provides examples of nonreflexive space curves $\xx$ with their tangent surface $\Tan(\xx)$ being non-reflexive as well, such that the image of a Frobenius map of a generic point of $\xx$ is in the osculating hyperplane to $\xx$ at such point.  

Another feature of the paper that is worthwhile to point out is that part of the results obtained here can be seen as generalizations or improvements of some known results that appear in \cite{Ar,GV0,HV}, as we point out along the text. 

\section{Background}\label{back}

Let $\fq$ be the finite field of order $q$, where $q=p^h$ with $h>0$ and $p>0$ is a prime number, and let $\xx \subset \p^n(\fqc)$ be a (projective, geometrically irreducible, nondegenerate, algebraic) curve defined over $\fq$, where $\fqc$ denotes the algebraic closure of $\fq$, with $n \geq 2$. Given an algebraic extension $\mathbb{H}$ of $\fq$, the function field of $\xx$ over $\mathbb{H}$ will be denoted by $\mathbb{H}(\xx)$. If $X_0,\ldots,X_n$ denote the projective coordinates of $\p^n(\fqc)$, the affine coordinate functions of $\xx$ are $x_0,\ldots,x_n$, where $x_i$ is the projection of $X_i/X_0$ in $\fq(\xx)$ for $i=0,\ldots,n$.  Let $P \in \xx$ be a nonsingular point. Then, there exists a sequence $(j_0(P), \ldots,j_n(P))$ of integers with $0 \leq j_0(P)<\cdots<j_n(P)$, which is defined by all the possible intersection multiplicities of $\xx$ and some hyperplane of $\p^n(\fqc)$ at $P$. Such sequence is  called the order sequence of $P$. From \cite[Section 1]{SV}, we know that except for finitely many points on $\xx$, this sequence is the same. It is called order sequence of $\xx$, and denoted by $(\varepsilon_0,\ldots, \varepsilon_n)$. If $(\varepsilon_0,\ldots, \varepsilon_n)=(0,\ldots,n)$, the curve $\xx$ is said to be a classical curve. Otherwise, $\xx$ is nonclassical. Given a nonsingular point $P \in \xx$, the osculating hyperplane to $\xx$ at $P$, denoted by $H_P(\xx) \subset \p^n(\fqc)$, is the unique hyperplane such that the intersection multiplicity of $H_P(\xx)$ and $\xx$ at $P$ is $j_n(P)$ (see \cite[Section 1]{SV}).

Let $\zeta \in \fqc(\xx)$ be a separating element and let $k$ be a non-negative integer. Given $f \in \fqc(\xx)$, the $k$-th Hasse derivative of $f$ with respect to $\zeta$ is denoted by $D_\zeta^{(k)}f$. The following result follows from \cite[Section 1]{SV}. 

\begin{prop}\label{oschypeq}
Suppose that $\xx$ is defined by the coordinate functions $x_0,\ldots,x_n \in \fq(\xx)$.
\begin{itemize}
\item[(a)] The order sequence of $\xx$ is the minimal sequence, with respect to the lexicographic order, for which
$$
\det\left(D_\zeta^{(\varepsilon_i)}x_j\right)_{0 \leq i,j \leq n} \neq 0,
$$
where $\zeta \in \fqc(\xx)$ is a separating element.
\item [(b)] Let $P \in \xx$ be a nonsingular point and let $t$ be a local parameter at $P$. The osculating hyperplane at $P$ is spanned by the points 
$
\big((D_t^{(j_i)}x_0)(P):\cdots: (D_t^{(j_i)}x_n)(P) \big),
$
 for $i=0,\ldots,n-1$.
\end{itemize}
\end{prop}

If $\xx$ is defined over $\fq$ by coordinate functions $x_0,\ldots,x_n \in \fq(\xx)$, by \cite[Proposition 2.1]{SV}, there exists a sequence of non-negative integers $(\nu_0,\ldots,\nu_{n-1})$, chosen minimally in the lexicographic order, such that
\begin{equation}\label{fncl}
\left|
  \begin{array}{ccc}
  x_0^q & \ldots & x_n^q \\
  D_\zeta^{(\nu_0)}x_0 & \ldots & D_\zeta^{(\nu_0)}x_n \\
   \vdots & \cdots & \vdots \\
  D_\zeta^{(\nu_{n-1})}x_0 & \cdots & D_\zeta^{(\nu_{n-1})}x_n
  \end{array}
  \right| \neq 0.
  \end{equation}
This sequence is called $\fq$-Frobenius order sequence of $\xx$. It does not depend on $\zeta$ and it is invariant under change of projective coordinates of $\xx$ (see \cite[Proposition 1.4]{SV}). The curve $\xx$ is called $\fq$-Frobenius classical if $(\nu_0,\ldots,\nu_{n-1})=(0,\ldots,n-1)$, and $\fq$-Frobenius nonclassical otherwise. From \cite[Proposition 2.1]{SV}, there exists $I \in \{1, \ldots,n\}$ for which $\{\nu_0,\ldots,\nu_{n-1}\}=\{\varepsilon_0,\ldots, \varepsilon_n\} \backslash \{\varepsilon_I\}$.

Let $(\p^n(\fqc))^{'}$ be the dual projective space. The strict Gauss map of $\xx$, denoted by $\gamma$, is the map defined in the nonsingular locus of $\xx$ by $\gamma(P)=(H_P(\xx))^{'} \in (\p^n(\fqc))^{'}$, where $L^{'}$ denotes the dual of a hyperplane $L$. The Zariski closure of the image of $\gamma$ is the strict dual of $\xx$, denoted by $\xx^{'}$. Note that $\gamma$ can be extended to a morphism of a nonsingular model of $\xx$. For simplicity, we will denote such morphism also by $\gamma$.  %The separable (respectively inseparable) degree of $\gamma$ will be denoted by $\deg_s \gamma$ (respectively $\deg_i \gamma$).

For a given $r>0$, the $\F_{p^r}$-Frobenius map $\Phi_{p^r}:\xx \lra \p^n(\fqc)$ is defined by $\Phi_{p^r}:P \mapsto P^{p^r}$. It induces a map from the function field $\fqc(\xx)$ onto $\fqc(\xx)^{p^r}=\{f^{p^r} \ | \ f \in \fqc(\xx)\}$. Therefore, $\Phi_{p^r}$ is purely inseparable of degree $p^r$. We point out that although $\Phi_q$ is purely inseparable, we have $\Phi_q: \xx \lra \xx$, as $\xx$ is defined over $\fq$.

Suppose that $p>n$ and $\varepsilon_n=p^r$ for some $r>0$. If $\xx$ is defined by the coordinate functions $x_0,\ldots,x_n$, then \cite[Theorem 7.65]{HKT} ensures that there exist $z_0,\ldots,z_n \in \fq(\xx)$ such that $z_i$ is a separating element for at least one $i$ and
\begin{equation}\label{e}
z_0^{p^r}x_0+\cdots+z_n^{p^r}x_n=0.
\end{equation}
Furthermore, for a generic point $P \in \xx$, the osculating hyperplane to $\xx$ at $P$ is defined by
$$
H_P(\xx):(z_0^{p^r}(P))X_0+\cdots+(z_n^{p^r}(P))X_n=0.
$$
Therefore, in such case we conclude that $\gamma=\Phi_{p^r} \circ \gamma_s=(z_0^{p^r}:\cdots:z_n^{p^r})$, where $\gamma_s=(z_0:\cdots:z_n)$ is a separable morphism. In particular, $\deg_s \gamma= \deg \gamma_s$ and $\deg_i \gamma =\deg \Phi_{p^r}=p^r$.

We now recall the definition of the conormal variety. Let $\kk$ be an arbitrary field and  let $V \subset \p^n(\kk)$ be a closed variety. The nonsingular locus of $V$ is denoted by $V_{\text{reg}}$. Let $T_PV$ be the tangent space to $V$ at $P \in V_{\text{reg}}$.  The conormal variety of $V$ is defined by
$$
C(V) =\overline{\{(P,H^{'}) \ | \ T_PV \subset H \}} \subset V \times (\p^n(\kk))^{'},
$$
where the overline denotes the Zariski closure of the set and $H$ is a hyperplane. The image of $C(V)$ in $(\p^n(\kk))^{'}$ under the second projection is denoted by $V^{*}$, and it is called the dual of $V$ (note that this coincides with the strict dual when $V$ is plane curve). The variety $V$ is said to be reflexive if $C(V)=C(V^*)$, and nonreflexive otherwise. A curve $\xx \subset \p^n(\fqc)$, with order sequence $(\varepsilon_0,\ldots,\varepsilon_n)$ is nonreflexive if, and only if, $\varepsilon_2>2$ or $p=2$, see \cite{HV} and \cite{He}.

In what follows, for a given curve $\xx$, we will denote by $\tilde{\xx}$ a fixed nonsingular model of $\xx$, and the points of $\tilde{\xx}$ will be seen as branches of $\xx$. 

\section{Duality for multi-Frobenius nonclassical curves}\label{mainl}

In this section, we will give a characterization of certain multi-Frobenius nonclassical curves in terms of their strict dual curves. With such characterization on hands, we are able to present sufficient criteria for the strict dual of some curves to have some interesting geometrical properties. We start with the following Proposition, which will be a key ingredient in the proof of one of the main results, namely, Theorem \ref{main-1}.

\begin{prop}\label{ordseqdua}
Let $\yy \subset \p^n(\overline{\kk})$, with $p>n$, be a nondegenerate curve defined over a perfect field $\kk$ of characteristic $p>0$ by the coordinate functions $y_0,\ldots,y_n \in \kk(\yy)$. Assume that there exist an integer $r>0$ and $u_0,\ldots,u_n \in \kk(\yy)$ such that 
\begin{equation}\label{rel1}
u_0^{p^{r}}y_0+u_1^{p^{r}}y_1+\cdots+u_n^{p^{r}}y_n=0.
\end{equation}
Then the hyperplane
$
H^{r}_P: (u_0(P))^{p^{r}}Y_0+(u_1(P))^{p^{r}}Y_1+\cdots+(u_n(P))^{p^{r}}Y_n=0
$
is such that $I(P, \yy \cap H^{r}_P) \geq p^{r}$ for a generic point $P \in \tilde{\yy}$, and the equality holds if, and only if,
\begin{equation}\label{rel2}
\left(D_\zeta^{(1)} u_0\right)^{p^{r}}y_0+\left(D_\zeta^{(1)} u_1\right)^{p^{r}}y_1+\cdots+\left(D_\zeta^{(1)} u_n\right)^{p^{r}}y_n\neq0,
\end{equation}
where $\zeta \in \kk(\yy)$ is a separating element.
\end{prop}
\begin{proof}
If none of the $u_k$ is a separating element, then we clearly have $I(P, \yy \cap H^{r}_P) > p^{r}$ for a generic point $P \in \tilde{\yy}$. Thus, we may assume that $u_k$ is a separating element for at least one $k$. Let $P \in \tilde{\yy}$ be a generic point. We may assume that $P$ is neither a zero nor a pole of the functions $y_k, u_k$ for $k=0,1,\ldots,n$, say
$$
y_k=\sum_{m=0}^{\infty} a_{km}t^{m} \ \ \ \text{ and } \ \ \ u_k=\sum_{m=0}^{\infty} b_{km}t^{m},
$$ 
where $t \in \overline{\kk}(\yy)$ is a local parameter at $P$. In particular,
\begin{equation}\label{derhas}
D_t^{(1)} u_k= b_{k1}+2b_{k2}t+\cdots,
\end{equation}
for $k=1,\ldots,n$. If $u_j$ is a separating element, then we also may assume that $P$ is unramified in the cover $\pi_j=(1:u_j):\tilde{\yy} \lra \p^1(\overline{\kk})$, hence $b_{j1} \neq 0$. Since $u_k(P)=b_{k0}$, we obtain  
$$
\sum_{k=0}^{n}(u_k-u_k(P))^{p^{r}}y_k=\left(\sum_{k=0}^{n}b_{k1}^{p^r }a_{k0}\right) t^{p^{r}}+\cdots.
$$
Using \eqref{rel1} we have
$$
I(P, \yy \cap H^{r}_P)=v_{P}\left(\sum_{k=0}^{n}(u_k(P))^{p^{r}}y_k\right)=v_{P}\left(\sum_{k=0}^{n}(u_k-u_k(P))^{p^{r}}y_k\right)\geq p^r,
$$
and equality holds if, and only if, $\sum\limits_{k=0}^{n}b_{k1}^{p^r }a_{k0} \neq 0$. In turn, from \eqref{derhas} we have that $\sum\limits_{k=0}^{n}b_{k1}^{p^r }a_{k0} \neq 0$ is equivalent to \eqref{rel2} for $\zeta=t$. Since $D_\zeta^{(1)}=\frac{dt}{d \zeta}D_t^{(1)}$ for all separating element $\zeta$, we clearly have that \eqref{rel2} does not depend on the choice of the separating element. This finishes the proof.
\end{proof}

\begin{rem}\label{gen765}
 In \cite[Theorem 1]{GV0}, Garcia and Voloch characterized linearly independent sets of fields $\kk$ with positive characteristic over the subfields $\kk^{p^r}=\{u^{p^r}  \ | \ u \in \kk\}$ in terms of generalized Wronskians. When restricted to algebraic function fields of one variable over a perfect full constant field of positive characteristic, Proposition \ref{ordseqdua} complements  \cite[Theorem 1]{GV0}. Furthermore, from the uniqueness of the osculating hyperplane to a curve at a point, Proposition \ref{ordseqdua} might be seen as a refinement of \cite[Theorem 7.65(iii)]{HKT}. Indeed,  let $\yy$ be a curve as in Proposition \ref{ordseqdua}. In \cite[Theorem 7.65(iii)]{HKT} it is stated that if $\varepsilon_{n-1}<p^r$ and \eqref{rel1} holds, then $\varepsilon_n=p^k$ for some $k \geq r$, and equality holds if, and only if, some $u_i$ is a separating element. In turn, $u_i$ being a separating element is equivalent to $D_\zeta^{(1)} u_i \neq 0$. Thus, if $\varepsilon_n=p^r$, then \eqref{rel2} must hold, otherwise we would have two osculating hyperplanes at a generic point, which is impossible. Conversely, if \eqref{rel2} holds, then we clearly have that $D_\zeta^{(1)} u_i \neq 0$ for some $i$, and then $u_i$ is separating, thence $\varepsilon_n=p^r$. In other words, \cite[Theorem 7.65(iii)]{HKT} can be seen as a version of Proposition \ref{ordseqdua} in the case that $H_P^{r}$ is a candidate to osculating hyperplane at $P$, while in Proposition \ref{ordseqdua} $H_P^{r}$ can have any intersection multiplicity with $\yy$ at $P$.
\end{rem}

The next example illustrates Proposition \ref{ordseqdua}.

\begin{ex}
Assume that $p>3$ and consider $\fqc(t)$ to be the function field of $\p^1(\fqc)$. Let $\xx=\phi(\p^1(\fqc)) \subset \p^3(\fqc)$ be the nondegenerate rational curve, where $\phi=(1:x:y:z)$ is defined by the parametrization
$$
x=t, \ \ y=-\frac{2}{3t^{p-1}} \ \ \text{ and } \ \ z=\left(\frac{2}{3t^{p-1}}\right)^{p^2}+\frac{t^{2p+1}}{3}. 
$$
One can check that $\xx$ is given by $\overline{\mathcal{S}_1} \cap \overline{{\mathcal{S}_2}}$, where
\begin{equation}\label{s1}
S_1: Y^{p^2}+X^{2p+1}+X^{3p}Y+Z=0
\end{equation}
and
\begin{equation}\label{s2}
S_2: 2+3X^{p-1}Y=0
\end{equation}
are affine surfaces. From \eqref{s1} we obtain $(y^{p})^{p}+(x^2)^px+(x^3)^py+z=0$ and from \eqref{s2} we obtain
$$
\left(D_t^{(1)}y^{p}\right)^p+\left(D_t^{(1)}x^{2}\right)^px+\left(D_t^{(1)}x^{3}\right)^py+\left(D_t^{(1)}1\right)^pz=0.
$$
For a generic $P \in \tilde{\xx}$, set $H_P:(y^{p}(P))^{p}+(x^2(P))^pX+(x^3(P))^pY+Z=0$. Then $I(P,H_P \cap \xx)=p+1$. Actually, the order sequence of $\xx$ is $(0,1,p,p+1)$.
\end{ex}

The following lemma will be used in the sequel of the paper. However, we point out that the results in it can be of independent interest.

\begin{lem}\label{multi}
Let $\xx \subset \p^n(\fqc)$ be a curve defined over $\fq$ with coordinate functions $x_0,\ldots,x_n$. Let $q_1<\ldots<q_s$ be powers of $q$, where $s\leq n$, and denote by $(\nu_{j,0},\ldots,\nu_{j,n-1})$ the $\F_{q_j}$-Frobenius order sequence of $\xx$, $j \in \{1,\ldots,s\}$. Then the following hold:
\begin{itemize}
\item[(a)] 
$$M_s:=\left(
\begin{array}{cccc}
x_0 & x_1 & \cdots & x_n\\
x_0^{q_1} & x_1^{q_1} & \cdots & x_n^{q_1}\\
\cdots & \cdots&\cdots&\cdots\\
x_0^{q_s} & x_1^{q_s} & \cdots & x_n^{q_s}
\end{array} \right)
$$
has rank smaller than $s+1$ if, and only if, there exists a linear subvariety $\p^{L}\subset\p^n(\fqc)$ of dimension $s-1$ such that for a generic point $P \in \xx$, we have that $\{P,\Phi_{q_1}(P),\ldots,\Phi_{q_s}(P)\} \subset \p^L$. In this case, if $(\varepsilon_0,\ldots,\varepsilon_n)$ denotes the order sequence of $\xx$ and $\{\Phi_{q_1}(P),\ldots,\Phi_{q_s}(P)\}$ is not in a proper linear subvariety of $\p^L$, then $\varepsilon_n \geq q_1$.
\item[(b)] Assume that $s<n$ and $\rank(M_s)=s+1$. Then, there exists a positive integer $c$ such that 
$$N_{c}(q_1,\ldots,q_s):=\left(
\begin{array}{cccc}
D_\zeta^{(c)} x_0&D_\zeta^{(c)} x_1 & \cdots & D_\zeta^{(c)} x_n\\
x_0 & x_1 & \cdots & x_n\\
x_0^{q_1} & x_1^{q_1} & \cdots & x_n^{q_1}\\
\cdots & \cdots&\cdots&\cdots\\
x_0^{q_s} & x_1^{q_s} & \cdots & x_n^{q_s}
\end{array} \right)
$$
has rank $s+2$, where $\zeta \in \fq(\xx)$ is a separating element. If $c$ is the minimal integer satisfying such a property, then for each $j$, there exists an integer $0<\ell_j\leq n-1$ such that $c=\nu_{j,\ell_j}$. Moreover, we have that $c>1$ implies $p| c$.
\end{itemize}
\end{lem}
\begin{proof}
\begin{itemize}
\item[(a)] The coordinates of a generic point $P \in \xx$ are given by $(x_0(P):\cdots:x_n(P))$, hence $\Phi_{q_j}(P)=(x_0(P)^{q_j}:\cdots:x_n(P)^{q_j})$. If $A:=\{P,\Phi_{q_1}(P),\ldots,\Phi_{q_s}(P)\}$, denote by $M_s(P)$ the $(s+1)\times(n+1)$ matrix where the lines are the elements of $A$. If $\rank(M_s)<s+1$, then $\rank(M_s(P))<s+1$, which gives that the elements of $A$ seen as vector in $\fqc^{n+1}$ are linearly dependent, and then $A$ defines a projective subspace of dimension at most $s-1$. Conversely, assume that for a generic point $P \in  \xx$, $P$ and $\Phi_{q_j}(P)$ are in $\p^L$ for $j \in \{1,\ldots,s\}$, where $\p^L$ has (projective) dimension $s-1$. Then $\rank(M_s(P)) <s+1$, and since $P$ is generic, $\rank(M_s)<s+1$. For the later assertion, let $M_s^{'}$ be the $(s+1)\times (s+1)$ matrix obtained by omitting the last $n-s$ columns of $M_s$. Denote by $u_i$ the minor obtained by omitting the first line and the $i$-th column of $M_s^{'}$. Note that $u_i=z_i^{q_1}$ for some $z_i \in \fqc(\xx)$ and $u_i \neq 0$ for at least one $i$. Thus $\det(M_s^{'})=0$ gives that $\sum\limits_{i=0}^{s}z_i^{q_1}x_i=0$. The assertion then follows from Proposition \ref{ordseqdua}.
\item[(b)] For simplicity, let us fix a $j$ and denote $(\nu_{j,0},\ldots,\nu_{j,n-1})=(\nu_0,\ldots,\nu_{n-1})$. Set $\phi=(x_0,\ldots,x_n)$, $\Phi_{q_i}(\phi)=(x_0^{q^i},\ldots,x_n^{q^i})$ for $i=1,\ldots,s$ and $D_\zeta^{(k)}(\phi)=(D_\zeta^{(k)}x_0,\ldots,D_\zeta^{(k)}x_n)$ seen as elements of $\fq(\xx)^{n+1}$, where $k>0$. Suppose that $N_{\nu_\ell}(q_1,\ldots,q_s)$ has rank $s+1$ for all $0<\ell \leq n-1$. Then, for $\ell=1,\ldots,n-1$, the vectors $D_\zeta^{(\nu_\ell)}(\phi)$ belong to the $\fq(\xx)$-subspace spanned by $\{\phi,\Phi_{q_1}(\phi),\ldots,\Phi_{q_s}(\phi)\}$, which has dimension $s+1$. Thus, the set $\{\phi, \Phi_{q_j}(\phi),D_\zeta^{(\nu_1)}(\phi),\ldots,D_\zeta^{(\nu_{n-1})}(\phi)\}$ is linearly dependent over $\fq(\xx)$, which contradicts the definition of the Frobenius order sequence of a curve. Therefore, the former statement follows. The minimality follows by an analogous argument. The proof of the last statement is analogous to \cite[Theorem 3.7]{AB}, and will be omitted.
\end{itemize}
\end{proof}

We say that a curve $\xx \subset \p^n(\fqc)$ is strange if there exists a point $P \in \p^n(\fqc)$ such that all tangent lines to smooth points of $\xx$ meet in $P$. It is straightforward to check that a curve is strange if and only if there exists an inseparable projection of $\xx$ from a point $P$ onto a curve in $\p^{n-1}(\fqc)$. In order to prove one of our main results, the following proposition will be needed. It provides a characterization of certain multi-Frobenius nonclassical curves with degenerate strict dual. Note that such result is a version of \cite[Proposition 2]{GV0} with the additional hypotheses of our situation.

\begin{prop}\label{nondeg}
Let $\xx \subset \p^n(\fqc)$ be a nonstrange curve defined over $\fq$ with order sequence $(0,1,\varepsilon_2,\ldots,\varepsilon_n)$ such that $\varepsilon_n=p^r$, with $p>n>2$. Let $1 \leq m_1< \cdots < m_{s}$ be a sequence of integers, where $1 \leq s \leq n-1$. The following conditions are equivalent:
\begin{itemize}
\item[(a)]  $\xx$ is $\F_{q^{m_i}}$-Frobenius nonclassical for $i=1,\ldots,s$  such that the last term of the $\F_{q^{m_i}}$-Frobenius order sequence equals to $\varepsilon_n$ and the strict dual $\xx^{'}$ of $\xx$ is degenerate.
\item[(b)] $\xx$ lies on a cone over a nondegenerate curve $\cc \subset \p^{n-1}(\fqc)$ defined over $\fq$ with orders $\lambda_0<\cdots<\lambda_{n-1}=\varepsilon_n$ such that for a generic point $Q \in \cc$,  $\{Q,\Phi_{q^{m_1}}(Q),\ldots,\\ \Phi_{q^{m_{s}}}(Q)\}\subset H_Q(\cc)$, where $H_Q(\cc)$ is the osculating hyperplane to $\cc$ at $Q$. 
\end{itemize}
\end{prop}
\begin{proof}
Let $x_0,x_1,\ldots,x_n$ be the projective coordinate functions of $\xx$. Since $\varepsilon_n=p^r$ there exist $z_0, \ldots z_n \in \fq(\xx)$ such that
\begin{equation}\label{eqnc}
z_0^{p^r}x_0+z_1^{p^r}x_1\cdots+z_n^{p^r}x_n=0.
\end{equation}
Equation \eqref{eqnc} implies that the osculating hyperplane to $\xx$ at $P \in \tilde{\xx}$ is defined by
$$
H_P(\xx):(z_0^{p^r}(P))X_0+\cdots+(z_n^{p^r}(P))X_n=0.
$$
Let us assume that (a) holds. From the fact that the last term of the $\F_{q^{m_i}}$-Frobenius order sequence equals to $\varepsilon_n$, for $i=1,\ldots,s$, together with Proposition \ref{oschypeq} and \eqref{fncl}, we have that $\Phi_{q^{m_i}}(P) \in H_P(\xx)$ for infinitely many points $P \in \tilde{\xx}$. Hence, we have the following equation in $\fq(\xx)$ for $i=1,\ldots,s$:
\begin{equation}\label{eq6}
z_0^{p^r}x_0^{q^{m_i}}+z_1^{p^r}x_1^{q^{m_i}}+\cdots+z_n^{p^r}x_n^{q^{m_i}}=0.
\end{equation} 
%which implies
%\begin{equation}\label{eq7}
%z_0x_0^{p^{h_i-r}}+z_1x_1^{p^{h_i-r}}+\cdots+z_nx_n^{p^{h_i-r}}=0.
%\end{equation}
Suppose that $\xx^{'}$ is degenerate. Then there exists a linear proper subvariety $\p^L \subset \p^n(\fqc)$ such that $\xx^{'}\subset \p^L$ is nondegenerate. Since $\xx^{'}$ is defined over $\fq$, so is $\p^L$. Thus there exist $a_0,\ldots,a_n \in \fq$ such that $\sum\limits_{j=0}^{n}a_jz_j=0$. After a projective change of coordinates defined over $\fq$, we may assume that $z_n=0$.
 
Since $\xx$ is nonstrange, the degeneracy of $\xx^{'}$ implies, via \cite[Main Theorem]{Ka} (see also \cite[Proposition 2]{GV0}), that $\xx$ lies on a cone over the curve $\cc \subset \p^{n-1}(\fqc)$ with projective coordinate functions $x_0,x_1,\ldots,x_{n-1}$ with order sequence $(\lambda_0,\ldots,\lambda_{n-1})$ such that $\lambda_{n-1}=\varepsilon_n=p^r$ (note that the nonstrangeness of $\xx$ guarantees that at least one of the $x_i$ is a separating element for $i=0,\ldots,n-1$). Now, let $\delta=(f_0:f_1: \cdots:f_{n-1})$ be the strict Gauss map of $\cc$. From the uniqueness of the osculating hyperplane to $\xx$ at a generic point $P \in \xx$, we conclude that
%$$
%\left(\frac{f_i}{f_0} \right)(P)=\left(\frac{z_i}{z_0} \right)^{p^r}(P)
%$$
%for infinitely many points $P \in \xx$, for $i=1,\ldots,n-1$. Hence 
$$
\frac{f_i}{f_0}=\left(\frac{z_i}{z_0} \right)^{p^r}
$$
for $i=1,\ldots,n-1$, which implies that $\delta=(z_0^{p^r}:\cdots:z_{n-1}^{p^r})$. Therefore, relation \eqref{eq6} holds in $\fqc(\cc)$, namely
\begin{equation}\label{eqcone}
z_0^{p^r}x_0^{q^{m_i}}+z_1^{p^r}x_1^{q^{m_i}}+\cdots+z_{n-1}^{p^r}x_{n-1}^{q^{m_i}}=0
\end{equation} 
for $i=1,\ldots,s$. However, \eqref{eqcone} implies that $\Phi_{q^{m_i}}(Q) \in H_Q(\cc)$ for infinitely many points $Q \in \cc$, which gives (b).

Conversely, assume that (b) holds. From \cite[Proposition 2]{GV0} we immediately have that $\xx^{'}$ is degenerate. Note that, since $\cc$ is defined over $\fq$, we may assume that it is obtained by projection of $\xx$ from the point $(0:\cdots:0:1)$. Thus, if $x_0,\ldots,x_{n-1} \in \fq(\cc)$ are the coordinate functions of $\cc$, then $x_0,\ldots,x_n$ are the coordinates of $\xx$ for some $x_n \in \fq(\xx)$. From $\lambda_{n-1}=p^r$ we have that there are $z_0,\ldots,z_{n-1} \in \fq(\cc)$, with $z_i$ separating for at least one $i$, such that $\sum\limits_{i=0}^{n-1}z_i^{p^r}x_i=0$, and the osculating hyperplane at a generic point $Q\in \cc$ is defined by $\sum\limits_{i=0}^{n-1}(z_i^{p^r}(Q))X_i=0$. Thus \eqref{eqcone} holds for for $i=1,\ldots,s$. Since \eqref{eqcone} and $\sum\limits_{i=0}^{n-1}z_i^{p^r}x_i=0$ hold in $\fq(\xx)$ as well and $\varepsilon_n=p^r$, we obtain (a).
\end{proof}

\begin{ex}
The following example pertain to Proposition \ref{nondeg}. It appears in \cite[Example 3.9]{Ar} with more details; we will reproduce it here for sake of completeness. Suppose $p>3$. Let $u,m$ be co-prime positive integers with $m >2$ and $m>u$, and consider the Borges curve defined over $\fq$ by $\ff:f(x,y)=0$, where
$$
f(x,y)=\frac{(x^{q^u}-x)(y^{q^m}-y)-(x^{q^m}-x)(y^{q^u}-y)}{(x^{q^2}-x)(y^{q}-y)-(x^{q}-x)(y^{q^2}-y)}.
$$
This curve is the only simultaneously $\F_{q^m}$- and  $\F_{q^u}$-Frobenius nonclassical curve for the morphism of lines (\cite[Theorem 1.1]{Bo2}), and its order sequence is $(0,1,q^u)$. %The Gauss map of $\ff$ is defined by $\delta=(xy^{q^u}-yx^{q^u}:-(y^{q^u}-y):x^{q^u}-x)$, and since $\ff$ is $\F_{q^u}$-Frobenius nonclassical, one has $\deg_s(\delta)=1$.
Let $\xx=\psi(\tilde{\ff})$, where $\psi=(1:x:y:xy)$. %The nonclassicality of $\ff$ gives $D_x^{(2)}(y)=0$. Thus $D_x^{(2)}(xy)=D_x^{(1)}(y)+xD_x^{(2)}(y)=D_x^{(1)}(y)$, and since $\ff$ is $\F_{q^r}$-Frobenius nonclassical for $r=u,m$, we obtain 
%$$
%D_x^{(1)}y=\frac{y^{q^u}-y}{x^{q^u}-x}=\frac{y^{q^m}-y}{x^{q^m}-x} \neq 0.
%$$
%Hence, a computation shows that 
The order sequence of $\xx$ is $(0,1,2,q^u)$ and the $\F_{q^r}$-Frobenius order sequence is $(0,1,q^u)$ for $r=u,m$. The Gauss map of $\xx$ is given by 
$$
\gamma=(x^{q^u}y^{q^m}-y^{q^u}x^{q^m}:-(y^{q^m}-y^{q^u}):x^{q^m}-x^{q^u}:0).
$$
  Thus $\gamma$ is purely inseparable, as $\fqc(\xx)=\fqc(\ff)$, and $\xx^{'}=\gamma(\xx)$ is degenerate. Note that $\xx$ lies on a cone over $\ff$, and $\{Q,\Phi_{q^u}(Q),\Phi_{q^m}(Q)\} \subset H_Q(\ff)$ for a generic point $Q \in \ff$.
\end{ex}

\begin{cor}\label{cordeg}
Let $\xx \subset \p^n(\fqc)$ be a nonstrange curve defined over $\fq$ with order sequence $(0,1,\varepsilon_2,\ldots,\varepsilon_n)$ such that $\varepsilon_n=p^r$, with $p>n>2$. Let $1 \leq m_1< \cdots < m_{n-1}$ be a sequence of integers such that $p^r<q^{m_1}$. If $\xx$ is $\F_{q^{m_i}}$-Frobenius nonclassical for $i=1,\ldots,n-1$  such that the last term of the $\F_{q^{m_i}}$-Frobenius order sequence equals to $\varepsilon_n$ and $\varepsilon_n<q^{m_j-m_{j-1}}$ for $j>1$, then $\xx^{'}$ is nondegenerate. 
\end{cor}
\begin{proof}
If $\xx^{'}$ is degenerate, then by Proposition \ref{nondeg} and Lemma \ref{multi}(a), for $h_0=0$, we have that $\det(x_j^{p^{h_i}})_{0 \leq i,j,\leq n-1}=0$. Let $s$ such that both  $s$ and $s+1$ last lines of $(x_j^{p^{h_i}})_{0 \leq i,j,\leq n-1}$ have rank $s$. Then again by Lemma \ref{multi}(a), we have that if $s=n-1$, then $\varepsilon_n \geq q^{m_1}$. If $s<n-1$, then $\varepsilon_n \geq q^{m_{n-s}-m_{n-s-1}}$.
\end{proof}

The main results of the section are the following ones. The former result shows, under certain conditions, how the multi-Frobenius nonclassicality of a nonplane curve and some geometric properties of its strict dual interact to each other. We remark that one of the implications of Theorem \ref{main-1}, namely (1) $\Rightarrow$ (2), is a generalization of the first statement of \cite[Proposition 7]{HV}.

\begin{thm}\label{main-1}
Let $\xx \subset \p^n(\fqc)$ with $n>2$ be a nonstrange curve defined over $\fq$ by the coordinate functions $x_0,\ldots,x_n$ with order sequence $(0,1,\varepsilon_2,\ldots,\varepsilon_n)$, where $p>n$. Let $1 \leq m_1< \cdots < m_{n-1}$ be a sequence of integers and let $h_i>0$ be such that $q^{m_i}=p^{h_i}$ for $i=1\ldots,n-1$. Moreover, denote by $\gamma^{'}:\xx^{'} \lra \xx^{''}$ the Gauss map of $\xx^{'}$. Assume that:
\begin{itemize}
\item The Gauss map $\gamma:\xx \lra \xx^{'}$ is purely inseparable and $\xx^{'}$ is nondegenerate.
\item $\varepsilon_n=p^r<q^{m_1}$. In particular, there are $z_0,\ldots,z_n \in \fq(\xx)$ such that $z_0^{p^r}x_0+\cdots+z_0^{p^r}x_0=0$.
\item For a generic point $P \in \xx$, $\{\Phi_{q^{m_1}}(P),\ldots,\Phi_{q^{m_{n-1}}}(P)\}$ is not in a linear projective subvariety of dimension $n-3$.
\end{itemize}
Then the following conditions are equivalent:
\begin{itemize}
\item[(1)] $\xx$ is $\F_{q^{m_i}}$-Frobenius nonclassical for $i=1,\ldots,n-1$ such that the last term of the $\F_{q^{m_i}}$-Frobenius order sequence equals to $\varepsilon_n$ and
\begin{equation}\label{hipder0}
\left(D_\zeta^{(1)} x_0\right)^{p^{h_i-r}}z_0+\left(D_\zeta^{(1)} x_1\right)^{p^{h_i-r}}z_1+\cdots+\left(D_\zeta^{(1)} x_n\right)^{p^{h_i-r}}z_n\neq0
\end{equation}
for $i=1,\ldots,n-1$, where $\zeta$ is a separating element of $\fqc(\xx)$.
\item[(2)] $\xx^{'}$ has order sequence $(\varepsilon_0^{'},\ldots,\varepsilon_n^{'})=(0,1,p^{h_1-r},\ldots,p^{h_{n-1}-r})$, $\xx^{''}=\xx$ with $\gamma^{'}\circ \gamma =\Phi_{p^{h_{n-1}}}$ and $\xx^{'}$ is $\F_{q^{m_{n-1}-m_j}}$-Frobenius nonclassical with the last term of the $\F_{q^{m_{n-1}-m_j}}$-Frobenius order sequence of $\xx^{'}$ equals to $p^{h_{n-1}-r}$ for $j=1,\ldots,n-2$.
\end{itemize}
\end{thm}
\begin{proof}
Assume that (1) holds. Using the same notation of Proposition \ref{nondeg}, we have that \eqref{eq6} hold for $i=1,\ldots,n-1$, which implies
\begin{equation}\label{equ7}
z_0x_0^{p^{h_i-r}}+z_1x_1^{p^{h_i-r}}+\cdots+z_nx_n^{p^{h_i-r}}=0. 
\end{equation}
 Since $\gamma$ is purely inseparable, the function field of $\xx$ is given by $\fq(z_1/z_0,\cdots,z_n/z_0)=\fq(\xx^{'})$, and $z_j/z_0$ is separating for at least one $j$. In particular, we have $x_0,\ldots,x_n \in \fq(\xx^{'})$ and, since \eqref{hipder0} holds for $i=1,\ldots,n-1$, then Proposition \ref{ordseqdua} gives that $p^{h_{i}-r}$ are orders of $\xx^{'}$ for $i=1,\ldots,n-1$, as $\xx^{'}$ is nondegenerate. From the fact that $\varepsilon_0=0$ and $\varepsilon_1=1$, we conclude that $(\varepsilon_0^{'},\ldots,\varepsilon_n^{'})=(0,1,p^{h_1-r},\ldots,p^{h_{n-1}-r})$.

By Proposition \ref{ordseqdua}, we have that the osculating hyperplane at a generic point $Q \in \tilde{\xx^{'}}$ is
$$
H_{Q}^{h_{n-1}-r}: (x_0(Q))^{p^{h_{n-1}-r}}Z_0+(x_1(Q))^{p^{h_{n-1}-r}}Z_1+\cdots+(x_n(Q))^{p^{h_{n-1}-r}}Z_n=0.
$$
Hence the Gauss map $\gamma^{'}:\xx^{'} \lra \xx^{''}$ is given by $\gamma^{'}=(x_0^{p^{h_{n-1}-r}}:x_1^{p^{h_{n-1}-r}}:\cdots:x_n^{p^{h_{n-1}-r}})$. Therefore, $\xx^{''}=\xx$ and $\gamma^{'}\circ \gamma =\Phi_{p^{h_{n-1}}}$. Finally, raising \eqref{equ7} to the power $p^{h_{n-1}-h_i}$ for $i=1,\ldots,n-2$, we obtain 
$$
z_0^{p^{h_{n-1}-h_i}}x_0^{p^{h_{n-1}-r}}+z_1^{p^{h_{n-1}-h_i}}x_1^{p^{h_{n-1}-r}}+\cdots+z_n^{p^{h_{n-1}-h_i}}x_n^{p^{h_{n-1}-r}}=0.
$$
Since $H_{Q}^{h_{n-1}-r}$ is the osculating hyperplane of $\xx^{'}$ at a generic point $Q$, we have that $\Phi_{p^{h_{n-1}-h_i}}(Q) \in H_Q^{h_{n-1}-r}$, which proves the last assertion of (2).

Conversely, assume that (2) holds. From \cite[Theorem 3]{GV0} we immediately obtain that $\xx$ is $\F_{q^{m_{n-1}}}$-Frobenius nonclassical such that the last term of the $\F_{q^{m_{n-1}}}$-Frobenius order sequence equals $\varepsilon_n$. Since $\varepsilon_n=p^r$, equation \eqref{eqnc} holds, where $\gamma=(z_0^{p^r}: \cdots:z_n^{p^r})$, and the $\F_{q^{m_{n-1}}}$-Frobenius nonclassicality as we just obtained gives
\begin{equation}\label{volta}
z_0x_0^{p^{h_{n-1}-r}}+z_1x_1^{p^{h_{n-1}-r}}+\cdots+z_nx_n^{p^{h_{n-1}-r}}=0.
\end{equation} 
Since $\varepsilon^{'}_n=p^{h_{n-1}-r}$, we conclude that the osculating hyperplane at a generic point $Q \in \tilde{\xx^{'}}$ is
$$
H_{Q}^{h_{n-1}-r}: (x_0(Q))^{p^{h_{n-1}-r}}Z_0+(x_1(Q))^{p^{h_{n-1}-r}}Z_1+\cdots+(x_n(Q))^{p^{h_{n-1}-r}}Z_n=0
$$
(here we once again use that $\gamma$ is purely inseparable). Thus from the hypothesis on the $\F_{q^{m_{n-1}-m_j}}$-Frobenius nonclassicality of $\xx^{'}$, we have that
$$
z_0^{p^{h_{n-1}-h_j}}x_0^{p^{h_{n-1}-r}}+z_1^{p^{h_{n-1}-h_j}}x_1^{p^{h_{n-1}-r}}+\cdots+z_n^{p^{h_{n-1}-h_j}}x_n^{p^{h_{n-1}-r}}=0
$$
for $j=1,\ldots,n-2$, which gives
\begin{equation}\label{volta2}
z_0x_0^{p^{h_{j}-r}}+z_1x_1^{p^{h_{j}-r}}+\cdots+z_nx_n^{p^{h_{j}-r}}=0.
\end{equation}
for all $j=1,\ldots,n-2$. Since equality above holds for $j=n-1$, we obtain that \eqref{eq6} holds for $i=1,\ldots,n-1$, and then $\xx$ is $\F_{q^{m_i}}$-Frobenius nonclassical for $i=1,\ldots,n-1$ such that the last term of the $\F_{q^{m_i}}$-Frobenius order sequence equals to $\varepsilon_n$. Now let $\mathcal{D}_2(Q)$ denote the $\fqc$-vector space of all hyperplanes in $(\p^n(\fqc))^{'}$ that intersect $\xx^{'}$ at $Q$ with multiplicity at least $2$. Then $H_{Q}^{h_{j}-r} \in \mathcal{D}_2(Q)$. Since $\xx^{'}$ is nondegenerate, it follows from \cite[Section 1]{SV} that $\dim_{\fqc}(\mathcal{D}_2(Q))=n-1$. We claim that $\{H_Q^{h_j-r} \ | \ j=1,\ldots,n-1\}$ is a basis of $\mathcal{D}_2(Q)$. Indeed, since $\{P,\Phi_{q^{m_2-m_1}}(P),\ldots,\Phi_{q^{m_{n-1}-m_1}}(P)\}$ is not in a linear projective subvariety of dimension $n-3$, we have by Lemma \ref{multi} that 
$$
\rank\left(
\begin{array}{cccc}
x_0& x_1 & \cdots & x_{n}\\
x_0^{p^{h_{2}-h_1}} & x_1^{p^{h_2-h_1}} & \cdots & x_{n}^{p^{h_2-h_1}}\\
\cdots & \cdots&\cdots&\cdots\\
x_0^{p^{h_{{n-1}}-h_1}} & x_1^{p^{h_{n-1}-h_1}} & \cdots & x_{n}^{p^{h_{n-1}-h_1}}
\end{array} \right)=n-1,
$$
which gives that, for a generic $Q \in \xx^{'}$, 
$$
\rank\left(
\begin{array}{cccc}
(x_0(Q))^{p^{h_{1}-r}} & (x_1(Q))^{p^{h_1-r}} & \cdots & (x_n(Q))^{p^{h_1-r}}\\
(x_0(Q))^{p^{h_{2}-r}} & (x_1(Q))^{p^{h_2-r}} & \cdots & (x_n(Q))^{p^{h_2-r}}\\
\cdots & \cdots&\cdots&\cdots\\
(x_0(Q))^{p^{h_{{n-1}}-r}} & (x_1(Q))^{p^{h_{n-1}-r}} & \cdots & (x_n(Q))^{p^{h_{n-1}-r}}
\end{array} \right)=n-1,
$$
and the claim follows. From this, we obtain that $I(Q,\xx^{'} \cap H_Q^{h_1-r})=p^{h_1-r}$, since otherwise we would have that $H_Q^{h_j-r} \in \mathcal{D}_{p^{h_2-r}}(Q)$ for $j=1,\ldots,n-1$, which is a contradiction since $\dim_{\fqc}(\mathcal{D}_{p^{h_2-r}}(Q))=n-2$. Proceeding in this way, we obtain that $I(Q,\xx^{'} \cap H_Q^{h_j-r})=p^{h_j-r}$ for $j=2,\ldots,n-2$, and the result follows from Proposition \ref{ordseqdua}.
\end{proof}

\begin{rem}\label{obs1}
\begin{enumerate}
\item Note that the hypothesis ``$\{\Phi_{q^{m_1}}(P),\ldots,\Phi_{q^{m_{n-1}}}(P)\}$ is not in a linear projective subvariety of dimension $n-3$ for a generic $P \in \xx$'' of Theorem \ref{main-1} is only used to prove (2) $\Rightarrow$ (1).
\item In Theorem \ref{main-1}, the dual $\xx^{'}$ of $\xx$ is assumed to be nondegenerate. In view of the characterization given in Proposition \ref{nondeg}, we know exactly which type of curves do not satisfy such condition. In particular, if we assume that $\varepsilon_n<\min\{q^{m_1},q^{m_j-m_{j-1}}\}$ for $j>1$, we obtain that $\xx^{'}$ is nondegenerate by Corollary \ref{cordeg}.
\item The strict duals of the curves satisfying the hypotheses of Theorem \ref{main-1} and (1) are nonreflexive.
\end{enumerate}
\end{rem}

We now state one of the implications of Theorem \ref{main-1} under weaker, but less technical conditions. However, note that in Theorem \ref{main0} the dual curve $\xx^{'}$ of $\xx$ is not assumed to be nondegenerate.

\begin{thm}\label{main0}
Let $\xx \subset \p^n(\fqc)$ be a nonstrange curve defined over $\fq$ with order sequence $(0,1,2,\ldots,n-1,\varepsilon_n)$ such that $\varepsilon_n>n$, with $p>n$. Let $1 \leq m_1< \cdots < m_{n-1}$, be a sequence of integers such that $\varepsilon_n<q^{m_1}$. Assume that $\xx$ is $\F_{q^{m_i}}$-Frobenius nonclassical for $i=1,\ldots,n-1$, the Gauss map $\gamma:\xx \lra \xx^{'}$ is purely inseparable and $q^{m_j-m_{j-1}}>\varepsilon_n$ for $j>1$. Then $\varepsilon_n=p^r$ for some $r>0$, the order sequence of the strict dual $\xx^{'}$ is $(\varepsilon_0^{'},\ldots,\varepsilon_n^{'})=(0,1,p^{h_1-r},\ldots,p^{h_{n-1}-r})$ and $\xx=\xx^{''}$ with $\gamma^{'} \circ \gamma =\Phi_{p^{h_{n-1}}}$, where  $q^{m_i}=p^{h_i}$, $i=1\ldots,n-1$. Moreover, if $n>2$, the last term of the $\F_{q^{m_{n-1}-m_j}}$-Frobenius order sequence of $\xx^{'}$ equals to $p^{h_{n-1}-r}$ for $j=1,\ldots,n-2$.
\end{thm}
\begin{proof}
The fact that $\varepsilon_n$ is a power of $p$ follows from $p$-adic criteria, see \cite[Corollary 1.9]{SV}. Since $\xx$ is $\F_{q^{m_i}}$-Frobenius nonclassical for $i=1,\ldots,n-1$ and the Frobenius order sequence is always a subsequence of the order sequence, we conclude that the $\F_{q^{m_i}}$-Frobenius order sequence of $\xx$ is $(0,1,\ldots,n-2,\varepsilon_n)$ for all $i \in \{1,\ldots,n-1\}$. Again, with notation of Proposition \ref{nondeg}, we have that \eqref{eq6} holds for $i=1,\ldots,n-1$, which implies
\begin{equation}\label{eq7}
z_0x_0^{p^{h_i-r}}+z_1x_1^{p^{h_i-r}}+\cdots+z_nx_n^{p^{h_i-r}}=0. 
\end{equation}
Since $q^{m_j-m_{j-1}}>\varepsilon_n$, we have from Corollary \ref{cordeg} that $\xx^{'}$ is nondegenerate. We aim to show that this also implies that 
\begin{equation}\label{hipder}
\left(D_\zeta^{(1)} x_0\right)^{p^{h_k-r}}z_0+\left(D_\zeta^{(1)} x_1\right)^{p^{h_k-r}}z_1+\cdots+\left(D_\zeta^{(1)} x_n\right)^{p^{h_k-r}}z_n\neq0
\end{equation}
for $k=1,\ldots,n-1$. Suppose on contrary that relation above does not hold for some such $k$. Recall that $\gamma$ is purely inseparable, hence $\fq(\xx)=\fq(z_1/z_0,\cdots,z_n/z_0)=\fq(\xx^{'})$, and $z_j/z_0$ is separating for at least one $j$. In particular, the equations \eqref{eq7} and 
$$
\left(D_\zeta^{(1)} x_0\right)^{p^{h_k-r}}z_0+\left(D_\zeta^{(1)} x_1\right)^{p^{h_k-r}}z_1+\cdots+\left(D_\zeta^{(1)} x_n\right)^{p^{h_k-r}}z_n=0
$$
hold in $\fq(\xx^{'})$, that is, $x_0,\ldots,x_n \in \fq(\xx^{'})$. Thus, for a generic point $Q \in \tilde{\xx^{'}}$, the hyperplane
$$
V_Q:\left(\left(D_\zeta^{(1)} x_0\right)(Q)\right)^{p^{h_k-r}}Z_0+\left(\left(D_\zeta^{(1)} x_1\right)(Q)\right)^{p^{h_k-r}}Z_1+\cdots+\left(\left(D_\zeta^{(1)} x_n\right)(Q)\right)^{p^{h_k-r}}Z_n=0
$$
is such that $I(Q,\xx^{'}\cap V_Q) \geq p^{h_k-r}$. In the same way, $I(Q,\xx^{'}\cap H_Q^{h_i-r}) \geq p^{h_i-r}$ for $i=1,\ldots,n-1$, where
$$
H_Q^{h_i-r}:(x_0(Q))^{p^{h_i-r}}Z_0+(x_1(Q))^{p^{h_i-r}}Z_1+\cdots+(x_n(Q))^{p^{h_i-r}}Z_n=0.
$$
Now, denote by $\mathcal{D}_{p^{h_1-r}}(Q)$ the $\fqc$-vector space of all hyperplanes in $(\p^{n}(\fqc))^{'}$ that intersect $\xx^{'}$ at $Q$ with multiplicity at least $p^{h_1-r}$. Since $\xx^{'}$ is nondegenerate, it follows from \cite[Section 1]{SV} that $\dim_{\fqc}\mathcal{D}_{p^{h_1-r}}(Q) \leq n-1$. Thus, we have that
$$\left(
\begin{array}{cccc}
(x_0(Q))^{p^{h_{1}-r}} & (x_1(Q))^{p^{h_1-r}} & \cdots & (x_n(Q))^{p^{h_1-r}}\\
(x_0(Q))^{p^{h_{2}-r}} & (x_1(Q))^{p^{h_2-r}} & \cdots & (x_n(Q))^{p^{h_2-r}}\\
\cdots & \cdots&\cdots&\cdots\\
(x_0(Q))^{p^{h_{{n-1}}-r}} & (x_1(Q))^{p^{h_{n-1}-r}} & \cdots & (x_n(Q))^{p^{h_{n-1}-r}}\\
\left(\left(D_\zeta^{(1)} x_0\right)(Q)\right)^{p^{h_k-r}}&\left(\left(D_\zeta^{(1)} x_1\right)(Q)\right)^{p^{h_k-r}}&\cdots&\left(\left(D_\zeta^{(1)} x_n\right)(Q)\right)^{p^{h_k-r}}
\end{array} \right)
$$
has rank $ \leq n-1$. In particular, since $Q \in \xx^{'}$ is generic, we obtain that $\rank(M_{n}^{h_k-h_1}(0,h_2-h_1,\ldots,h_{n-1}-h_1))\leq n-1$, where $M_{n}^{h_k-h_1}(0,h_2-h_1,\ldots,h_{n-1}-h_1):=$
$$\left(
\begin{array}{cccc}
x_0& x_1 & \cdots & x_{n}\\
x_0^{p^{h_{2}-h_1}} & x_1^{p^{h_2-h_1}} & \cdots & x_{n}^{p^{h_2-h_1}}\\
\cdots & \cdots&\cdots&\cdots\\
x_0^{p^{h_{{n-1}}-h_1}} & x_1^{p^{h_{n-1}-h_1}} & \cdots & x_{n}^{p^{h_{n-1}-h_1}}\\
\left(D_\zeta^{(1)} x_0\right)^{p^{h_k-h_1}}&\left(D_\zeta^{(1)} x_1\right)^{p^{h_k-h_1}}&\cdots&\left(D_\zeta^{(1)} x_{n}\right)^{p^{h_k-h_1}}
\end{array} \right).
$$
As a matter of fact, $\rank(M_{n}^{h_k-h_1}(0,h_2-h_1,\ldots,h_{n-1}-h_1))= n-1$. Indeed, if the $n-1$ first lines of such matrix were linearly dependent over $\fq(\xx)$, then we would conclude, arguing as in the proof of Corollary \ref{cordeg}, that $\varepsilon_n>q^{m_j-m_{j-1}}$ for some $j>1$. First, assume that $k=1$, i.e., $\rank(M_{n}^{0}(0,h_2-h_1,\ldots,h_{n-1}-h_1))= n-1$.
%$$\rank\left(
%\begin{array}{cccc}
%x_0^{p^{h_{{n-1}}-h_1}} & x_1^{p^{h_{n-1}-h_1}} & \cdots & x_n^{p^{h_{n-1}-h_1}}\\
%\cdots & \cdots&\cdots&\cdots\\
%x_0^{p^{h_{2}-h_1}} & x_1^{p^{h_2-h_1}} & \cdots & x_n^{p^{h_2-h_1}}\\
%x_0& x_1 & \cdots & x_n\\
%D_\zeta^{(1)} x_0&D_\zeta^{(1)} x_1&\cdots&D_\zeta^{(1)} x_n
%\end{array} \right) = n-1.
%$$
Then we conclude from Lemma \ref{multi}(b) that $\nu_{n-1}\geq p$, where $(\nu_0,\ldots,\nu_{n-1})$ is the $\F_{p^{h_j-h_1}}$-Frobenius order sequence of $\xx$, with $j \in \{2,\ldots,n-1\}$. Moreover, the same Lemma \ref{multi}(b) gives that
	$$\rank\left(
\begin{array}{cccc}
x_0^{p^{h_{{n-1}}-h_1}} & x_1^{p^{h_{n-1}-h_1}} & \cdots & x_n^{p^{h_{n-1}-h_1}}\\
\cdots & \cdots&\cdots&\cdots\\
x_0^{p^{h_{2}-h_1}} & x_1^{p^{h_2-h_1}} & \cdots & x_n^{p^{h_2-h_1}}\\
x_0& x_1 & \cdots & x_n\\
D_\zeta^{(j)} x_0&D_\zeta^{(j)} x_1&\cdots&D_\zeta^{(j)} x_n
\end{array} \right) = n-1
$$
for $j=1,\ldots,p-1$. Since $p>n$, this implies that
$$
\rank\left(D_\zeta^{(j)}x_i\right)_{\substack{0 \leq i \leq n \\ 0 \leq j \leq n-1}} \leq n-1,
$$
which implies that $\varepsilon_{n-1}>n-1$, a contradiction. Now, suppose that $k>1$. If $\rank(M_{n}^{h_k-h_1}(h_2-h_1,\ldots,h_{n-1}-h_1))= n-1$, then arguing as in the proof of Lemma \ref{multi}(a), we obtain that $\varepsilon_n \geq p^{h_2-h_1}=q^{m_2-m_1}$, which is again a contradiction. If 
$\rank(M_{n}^{h_k-h_1}(h_2-h_1,\ldots,h_{n-1}-h_1))< n-1$, then $\rank(M_{n}^{h_k-h_2}(0,h_3-h_2,\ldots,h_{n-1}-h_2))= n-2$. Thus, if $k=2$, arguing as in the case $k=1$ we obtain $\varepsilon_{n-2}>n-2$, arriving again at a contradiction. If $k>2$ and $\rank(M_{n}^{h_k-h_2}(h_3-h_2,\ldots,h_{n-1}-h_2))= n-2$, we obtain $\varepsilon_n \geq q^{m_3-m_2}$. If such rank is $<n-2$, we may repeat this process as many times as necessary, obtaining that either in $\varepsilon_n \geq q^{m_j-m_s}$ with $j>s$ or $\varepsilon_{j}>j$ for a $j<n$. In both cases, we arrive at a contradiction. Hence \eqref{hipder} holds for $k=1,\ldots,n-1$. Therefore, the order sequence of $\xx^{'}$ is $(0,1,p^{h_1-r},\ldots,p^{h_{n-1}-r})$ (recall that we always have $\varepsilon_1=1$).

The remaining assertions follows exactly as in the proof of Theorem \ref{main-1}.
	
\end{proof}

\begin{ex}\label{ex2}
Consider the geometrically irreducible plane curve $\cc$ defined over $\fq$ by the affine equation
\begin{equation}\label{eq9}
w^{q^2}-w=x^{q^3+1}-x^{q^2+q}.
\end{equation}
Note that $x$ is a separating element of $\fq(\cc)=\fq(x,w)$. Defining $y:=w-x^{q+1}$, we obtain that $\fq(x,y)=\fq(x,w)$ and
\begin{equation}\label{eq10}
y^{q^2}-y+(x^{q^2}-x)(x^{q^3}+x^q)=0,
\end{equation}
which is equivalent to
\begin{equation}\label{eq11}
(y-y^{q^4})x^{q^3}+(x-x^{q^4})x^{2q^3}=(y-y^{q^2})x^{q}+(x-x^{q^2})x^{2q}+(y-y^{q^2})^{q^2}x^{q^3}+(x-x^{q^2})^{q^2}x^{2q^3}.
\end{equation}
Now, let $z$ be an element in an extension of $\fq(\cc)$ (eventually in $\fq(\cc)$) such that
\begin{equation}\label{eq12}
z-z^{q^2}+(y-y^{q^2})x^q+(x-x^{q^2})x^{2q}=0
\end{equation}
and denote by $\ff$ the curve defined in $\p^3(\fqc)$ over $\fq$ by the coordinate functions $1,x,y,z$. Via \eqref{eq11}, we obtain
\begin{equation}\label{eq13}
z-z^{q^4}+(y-y^{q^4})x^{q^3}+(x-x^{q^4})x^{2q^3}=0.
\end{equation}
Now, let $\yy=\psi(\tilde{\ff}) \subset \p^3(\fqc)$, where $\psi=(1:x:x^2:z^q+y^qx+x^{q+2})$. The order sequence of $\yy$ is $(0,1,2,q)$. Moreover, \eqref{eq12} and \eqref{eq13} imply that $\yy$ is $\F_{q^2}$ and $\F_{q^4}$-Frobenius nonclassical with $\F_{q^i}$-Frobenius order sequence $(0,1,q)$ for $i=2,4$. Finally, the Gauss map of $\yy$ is $\gamma=(z^q:y^q:x^q:-1)$, which is purely inseparable, and $\yy^{'}$ is nondegenerate. Thus $\yy$ fulfills the hypotheses of Theorem \ref{main0}. A direct computation using Hasse derivatives and some of the equations above shows that the order sequence of $\yy^{'}$ is indeed $(0,1,q,q^3)$, and the $\F_{q^2}$-Frobenius order sequence of $\yy^{'}$ is $(0,1,q^3)$.
\end{ex}

The next result shows that we can improve \cite[Theorem 3.7]{Ar} for nonstrange curves with order sequence $(0,1,2,\ldots,n-1,\varepsilon_n)$ in which $\varepsilon_n<q$. 

\begin{cor}
Let $\xx \subset \p^n(\fqc)$ be a nonstrange curve defined over $\fq$ with order sequence $(0,1,2,\ldots,n-1,\varepsilon_n)$ such that $n<\varepsilon_n<q=p^h$, with $p>n$.  Suppose that $\xx$ is $\F_{q^{i}}$-Frobenius nonclassical for $i=1,\ldots,n-1$. Then $\varepsilon_n=p^r$ for some $r>0$, the strict Gauss map $\gamma$ is purely inseparable, the order sequence of the strict dual $\xx^{'}$ is $(0,1,p^{h-r},\ldots,p^{(n-1)h-r})$ and $\xx=\xx^{''}$ with $\gamma^{'} \circ \gamma =\Phi_{q^{(n-1)}}$. Furthermore, if $n>2$, the last term of the $\F_{q^{n-1-j}}$-Frobenius order sequence of $\xx^{'}$ equals to $p^{(n-1)h-r}$ for $j=1,\ldots,n-2$.
\end{cor}
\begin{proof}
Under these hypotheses, \cite[Theorem 3.7]{Ar} provides that the Gauss map $\gamma$ of $\xx$ is purely inseparable. The result now is a direct consequence of Theorem \ref{main0}.
\end{proof}

\begin{rem}
We finish this section by pointing out that the results presented here can be very useful in some situations. In fact, as we saw in Example \ref{ex2}, with the help of Theorem 3.10, it is possible to obtain the order sequence and Frobenius order sequence of some curves without performing long computations, as required by the usual methods. In turn, as mentioned in the introduction, obtaining such sequences is a crucial step to bound the number of rational points on the curve via Sth\"or-Voloch Theory, introduced in \cite{SV}, and variation of such method (see \cite{AB} for instance). We will exploit this situation in the following example.
\end{rem}

\begin{ex}
For $p>3$, let $\xx \subset \p^3(\fqc)$ be a curve defined over $\fq$ by affine coordinate functions $1,x,y,z$, where $q=p^h$, with order sequence $(\varepsilon_0,\varepsilon_1,\varepsilon_2,\varepsilon_3)$ such that $\varepsilon_3 \neq p^{2h-r}$ for some $r<h$. Assume that there exists a separating element $u \in \fq(\xx)$ such that 
\begin{equation}\label{fncex} 
(z-z^q)+u^{p^{h-r}}(y-y^q)+u^{2p^{h-r}}(x-x^q)=0.
\end{equation}
Now, set $\psi=(1:u:u^2:z^{p^r}+y^{p^r}u+x^{p^r}u^2):\tilde{\xx} \lra \p^3(\fqc)$. Since $u$ is a separating element, we have that $\psi(\tilde{\xx}) \subset \p^{3}(\fqc)$ is nonstrange. One can check that the following hold (see the proof Proposition \ref{dualforc}): $\psi$ is birational, $\psi(\tilde{\xx})$ is nondegenerate, $(\psi(\tilde{\xx}))^{'}$ is projectively equivalent to $\xx$, the order sequence and $\fq$-Frobenius order sequence of $\psi(\tilde{\xx})$ are  $(0,1,2,p^r)$ and $(0,1,p^r)$, respectively. We claim that $\psi(\tilde{\xx})$ is $\F_{q^2}$-Frobenius classical. In fact, if otherwise holds, we have that the $\F_{q^2}$-Frobenius order sequence of $\psi(\tilde{\xx})$ equals $(0,1,q)$. Then, by Theorem \ref{main0} we have that the order sequence of $\xx$ is $(0,1,p^{h-r},p^{2h-1})$, which is a contradiction to $\varepsilon_3 \neq p^{2h-1}$.

Therefore, we obtain via Sth\"or-Voloch Theory (more precisely, from \cite[Proposition 2.4 and proof of Theorem 2.13]{SV}) that the number $N_{q^{2}}(\tilde{\xx})$ of $\F_{q^{2}}$-rational points of $\tilde{\xx}$ is bounded by
\begin{equation}\label{cota1}
N_{q^{2}}(\tilde{\xx}) \leq \frac{6(g-1)+(q^{2}+3)\cdot \deg\psi({\tilde{\xx}})}{p^r},
\end{equation}
where $g$ denotes the genus of $\xx$. Thus, we obtained a bound for $N_{q^{2}}(\tilde{\xx})$ without having the full information of the order sequence of $\xx$ on hands. The price for that is computing $\deg\psi({\tilde{\xx}})$, which sometimes can be less difficult than computing such order sequence. 
\end{ex}

\section{Construction of nonreflexive multi-Frobenius nonclassical curves}

In this section we give a flavor of how we can construct curves having interesting properties with the help of the results and ideas of Section \ref{mainl}. It should be noted that the ideas applied here are not exhaustive; they certainly can be used in future work, specially for curves that are dual of curves with some known properties (Proposition \ref{dualforc}  and Corollary \ref{ultcor} will illustrate this).

  Recall that a curve $\xx$ over a field of characteristic $p>2$ with order sequence $(\varepsilon_0,\ldots,\varepsilon_n)$ is said to be nonreflexive if $\varepsilon_2>2$.  Rephrasing Theorem \ref{main0} in the following way may give a recipe to construct nonreflexive multi-Frobenius nonclassical curves in suitable cases. The definition of $N_{c}(q_1,\ldots,q_{s})$ appears in Lemma \ref{multi}.

\begin{prop}\label{aplic1}
Let $\xx \subset \p^n(\fqc)$ be a nonstrange curve defined over $\fq$ by coordinate functions $x_0,\ldots,x_m \in \fq(\xx)$ with order sequence $(0,1,2,\ldots,n-1,\varepsilon_n)$ such that $\varepsilon_n>n$, with $p>n>2$. Let $1 \leq m_1< \cdots < m_{n-1}$, be a sequence of integers such that $\varepsilon_n<\min\{q^{m_1},q^{m_j-m_{j-1}}\}$ for $j>1$. If the Gauss map of $\xx$ is purely inseparable and
$$
\det(N_{1}(q^{m_1},\ldots,q^{m_{n-1}}))=0,
$$
then $\varepsilon_n$ is a power of $p$, the strict dual curve $\xx^{'}$ of $\xx$ has order sequence $(0,1,q^{m_1}/\varepsilon_n,\ldots,\\ q^{m_{n-1}}/\varepsilon_n)$ and $\xx=\xx^{''}$ with $\gamma^{'} \circ \gamma =\Phi_{q^{m_{n-1}}}$. Moreover, the last term of the $\F_{q^{m_{n-1}-m_j}}$-Frobenius order sequence of $\xx^{'}$ equals to $q^{m_{n-1}}/\varepsilon_n$ for $j=1,\ldots,n-2$.
\end{prop}
\begin{proof}
 Since $\det(N_{1}(q^{m_1},\ldots,q^{m_{n-1}}))=0$ and $\varepsilon_n<\min\{q^{m_1},q^{m_j-m_{j-1}}\}$ for $j>1$, we conclude from Lemma \ref{multi} that $\xx$ is $\F_{q^{m_i}}$-Frobenius nonclassical for $i=1,\ldots,n-1$. Hence, the result follows immediately from Theorem \ref{main0}.
\end{proof}

Given a field $\kk$ of characteristic $p>0$, let $\xx \subset \p^n(\kk)$ be a curve defined over $\kk$. Denote by $\xx_{\text{reg}} \subseteq \xx$ the set of nonsingular points of $\xx$ and, given $P \in \xx_{\text{reg}}$, let $L_1(P)$ denote the tangent line to $\xx$ at $P$. The tangent variety of $\xx$ is defined by
$$
\Tan(\xx)=\overline{\bigcup_{P \in \xx_{\text{reg}}}L_1(P)}.
$$
In \cite[Theorem 3.3]{Ho} Homma provides a characterization of reflexive curves such that $\Tan(\xx)$ is nonreflexive. In \cite{Ho2}, the same author refines his result for space curves, i.e., curves embedded in $\p^3(\kk)$. More precisely, it is proven that the (non)reflexivity of $\xx$ and the tangent surface $\Tan(\xx)$ can be obtained only with the knowledge of the order sequence of $\xx$, see \cite[Theorem 0.1]{Ho2}. With this later result in mind, we are able to give the following.

\begin{cor}\label{cortan}
Let $\xx \subset \p^3(\fqc)$ be a nonstrange curve defined over $\fq$ by coordinate functions $x_0,x_1,x_2,x_3 \in \fq(\xx)$ with order sequence $(0,1,2,\varepsilon_3)$ such that $\varepsilon_3>3$, with $p>3$. Let $1 \leq m_1<  m_{2}$, be a sequence of integers such that $\varepsilon_3<\min\{q^{m_1},q^{m_2-m_1}\}$. If the Gauss map of $\xx$ is purely inseparable and
$$
\det(N_{1}(q^{m_1},q^{m_2}))=0,
$$
then $\varepsilon_3$ is a power of $p$, the order sequence of the strict dual $\xx^{'}$ of $\xx$ is $(0,1,q^{m_1}/\varepsilon_3, q^{m_2}/\varepsilon_3)$ and $\xx=\xx^{''}$ with $\gamma^{'} \circ \gamma =\Phi_{q^{m_{2}}}$. In particular,  $\xx^{'}$ and its tangent surface $\Tan(\xx^{'})$ are nonreflexive. Furthermore, the  $\F_{q^{m_{2}-m_1}}$-Frobenius order sequence of $\xx^{'}$ is $(0,1,q^{m_2}/\varepsilon_3)$.
\end{cor}
\begin{proof}
This follows directly from Proposition \ref{aplic1} and \cite[Theorem 0.1]{Ho2}.
\end{proof}

We now give a more explicit construction. Consider the curves $\ff$ and $\yy$ constructed in Example \ref{ex2}. Note that the strict dual $\yy^{'}$ is projectively equivalent to $\ff$. Hence, the order sequence of $\ff$ is $(0,1,q,q^3)$. The following result generalizes Example \ref{ex2}.

\begin{prop}\label{dualforc}
Let $r,s_1,\ldots,s_{n-1}$ be positive integers, with $s_1<s_2<\cdots<s_{n-1}$ such that $p^{s_i+r}$ is a power of $q$ for all $i=1,\ldots,n-1$. Assume that $\xx \subset \p^n(\fqc)$, with $p>n$, is a nondegenerate curve defined over $\fq$ by the coordinate functions $1,x_1,\ldots,x_n$, where $x_1$ is a separating element of $\fq(\xx)$ such that $\xx$ is contained in a nonlinear component of $\overline{S_n}$, where $S_n:=$
$$
\left|
\begin{array}{ccccccc}
1 & X_1^{p^{s_{n-1}+r}} & X_1^{2p^{s_{n-1}+r}} &  \cdots & X_1^{(n-2)p^{s_{n-1}+r}}&X_1^{(n-1)p^{s_{n-1}+r}}&\left(\sum\limits_{j=0}^{n-1}X_{n-j}^{p^r}X_1^{j}\right)^{p^{s_{n-1}+r}}\\
\cdots & \cdots&\cdots&\cdots& \cdots\\
1 & X_1^{p^{s_{1}+r}} & X_1^{2p^{s_{1}+r}} &  \cdots&X_1^{(n-2)p^{s_{1}+r}} & X_1^{(n-1)p^{s_{1}+r}}&\left(\sum\limits_{j=0}^{n-1}X_{n-j}^{p^r}X_1^{j}\right)^{p^{s_{1}+r}}\\
1 & X_1 & X_1^{2} &  \cdots&X_1^{n-2} & X_1^{n-1} & \sum\limits_{j=0}^{n-1}X_{n-j}^{p^r}X_1^{j}\\
0 & 0 & 0& \cdots& 0 & 1 & X_1^{p^{r}}
\end{array} \right| =0.
$$ 
Then $\xx$ has order sequence $(0,1,p^{s_1},p^{s_2},\ldots,p^{s_{n-1}})$, and if $n>2$, the last term of the $\F_{p^{s_{n-1}-s_j}}$-Frobenius order sequence of $\xx$ equals to $p^{s_{n-1}}$ for $j=1,\ldots,n-2$.
\end{prop}
\begin{proof}
Define $$\psi=(1:x_1:x_1^2:\cdots:x_1^{n-1}:x_{n}^{p^r}+x_{n-1}^{p^r}x_1+x_{n-2}^{p^r}x_1^2+\cdots+x_{2}^{p^r}x_1^{n-2}+x_1^{p^r+n-1}).$$ We clearly have that the map $\psi$ is birational and $\psi(\tilde{\xx})$ is nondegenerate and nonstrange. The Gauss map of $\psi(\tilde{\xx})$ is $\gamma=(x_n^{p^r}:x_{n-1}^{p^r}:\cdots:x_1^{p^r}:-1)$. Thus, $\gamma$ purely inseparable, and $\xx$ projectively equivalent to the strict dual of $\psi(\tilde{\xx})$  (in particular, the strict dual of $\psi(\tilde{\xx})$ is nondegenerate). Via Proposition \ref{oschypeq}(a), it follows that the order sequence of $\psi(\tilde{\xx})$ is $(0,1,\ldots,n-1, p^r)$. Since $\xx \subset \overline{S_n}$ and $x_1$ is a separating element, we conclude from Lemma \ref{multi} that $\psi(\tilde{\xx})$ is $\F_{p^{s_i+r}}$-Frobenius nonclassical for $i=1,\ldots,n-1$. Since the $\F_{p^{s_i+r}}$-Frobenius order sequence is a sub sequence of $(0,1,\ldots,n-1, p^r)$, we conclude that the $\F_{p^{s_i+r}}$-Frobenius order sequence of $\psi(\tilde{\xx})$ is $(0,1,\ldots,n-2, p^r)$ for $i=1,\ldots,n-1$. Now, set $$f:=x_{n}^{p^r}+x_{n-1}^{p^r}x_1+x_{n-2}^{p^r}x_1^2+\cdots+x_{2}^{p^r}x_1^{n-2}+x_1^{p^r+n-1}$$ and $x:=x_1$. From
$$
x_n+x_{n-1}x^{p^{s_i}}+x_{n-2}x^{2p^{s_i}}+\cdots+x^{(n-1)p^{s_i}+1}-f^{p^{s_i}}=0 
$$
for $i=1,\ldots,n-1$, we obtain, for all such $i$,
\begin{equation}\label{auxiliares}
D_x^{(1)}x_n+D_x^{(1)}x_{n-1}x^{p^{s_i}}+D_x^{(1)}x_{n-2}x^{2p^{s_i}}+\cdots+x^{(n-1)p^{s_i}}=0.
\end{equation}
Assume that, for some $j \in\{1,\ldots,n-1\}$, we have
$$
x_{n-1}+2x_{n-2}x^{p^{s_j}}+\cdots+(n-1)x^{(n-2)p^{s_j}+1}-(D_x^{(1)}f)^{p^{s_j}}=0.
$$
Applying $D_x^{(1)}$ in both sides of the last equality, we obtain
\begin{equation}\label{auxiliar2}
D_x^{(1)}x_{n-1}+2D_x^{(1)}x_{n-2}x^{p^{s_j}}+\cdots+(n-1)x^{(n-2)p^{s_j}}=0.
\end{equation}
We then have $n$ linear equations with $n-1$ variables over the the field $\fq(x)^{p^{s_1}}$ with a solution, namely $(D_x^{(1)}x_2,\ldots,D_x^{(1)}x_n)$. This gives that
$$\left|
\begin{array}{ccccc}
1 & x^{p^{s_1}} & x^{2p^{s_1}} &  \cdots & x^{(n-1)p^{s_1}}\\
1 & x^{p^{s_2}} & x^{2p^{s_2}} &  \cdots & x^{(n-1)p^{s_2}}\\
\cdots & \cdots&\cdots&\cdots& \cdots\\
1 & x^{p^{s_{n-1}}} & x^{2p^{s_{n-1}}} &  \cdots & x^{(n-1)p^{s_{n-1}}}\\
0 & 1 & 2x^{p^{s_j}}& \cdots & (n-1)x^{(n-2)p^{s_j}}
\end{array} \right| =0.
$$
Hence, arguing by induction on $n$ and using Lemma \ref{multi}, we conclude that, for some $r \leq n-1$, the normal rational curve of $\p^r(\fqc)$ defined by the morphism $(1:x:x^2:\cdots:x^{r}):\p^1(\fqc) \lra \p^{r}(\fqc)$ is nonclassical, which is a contradiction since $p>n$. This means that, for all $j$, \eqref{auxiliar2} does not hold. The result then follows from Remark \ref{obs1} and Theorem \ref{main-1}.
\end{proof}

%Proposition \ref{dualforc} for $n=3$ reads as follows.

\begin{cor}\label{ultcor}
Let $r,s_1,s_{2}$ be positive integers, with $s_1<s_2$ and $p^{s_i+r}$ being a power of $q$ for $i=1,2$. Assume that $\xx \subset \p^3(\fqc)$, with $p>3$, is a nondegenerate curve defined over $\fq$ by the coordinate functions $1,x,y,z$, where $x$ is a separating element of $\fq(\xx)$ such that $\xx$ is contained in a nonlinear component of $\overline{S_3}$, where
$$
S_3:=\left|
\begin{array}{cccc}
1 & X^{p^{s_{2}+r}} & X^{2p^{s_{2}+r}} &\left(Z^{p^r}+Y^{p^r}X+X^{p^{r}+2}\right)^{p^{s_{2}+r}}\\
1 & X^{p^{s_{1}+r}} & X^{2p^{s_{1}+r}} &\left(Z^{p^r}+Y^{p^r}X+X^{p^{r}+2}\right)^{p^{s_{1}+r}}\\
1 & X & X^{2} &Z^{p^r}+Y^{p^r}X+X^{p^{r}+2}\\
0 & 0 & 1 & X^{p^{r}}
\end{array} \right| =0.
$$
Then the order sequence of $\xx$ is $(0,1,p^{s_1},p^{s_2})$. In particular, $\xx$ and its tangent surface $\Tan(\xx)$ are nonreflexive. Moreover, the $\F_{p^{s_2-s_1}}$-Frobenius order sequence of $\xx$ is $(0,1,p^{s_2})$.
\end{cor}

\begin{rem}
Although there are infinitely many curves in a nonlinear component of $\overline{S_n}$, only a finite number of such curves fulfill the condition that $x_1$ is a separating element. For instance, for $p=5$, $r=s_1=1$ and $s_2=3$, we obtain only one curve satisfying the hypotheses of Corollary \ref{ultcor}. As a matter of fact, it is the same curve obtained in Example \ref{ex2} for these parameters. However, it is not difficult to see that there are other constructions, similar to that of Proposition \ref{dualforc}, to obtain curves with the same desired properties. 
\end{rem}

We end this section by calling the attention to the fact that some of the results presented here can have their hypotheses modified or adapted in specific situations, specially in low dimensional projective spaces. To illustrate this, we will present a version Theorem \ref{main0} that also applies to nonreflexive space curves with nonreflexive tangent surfaces.

\begin{prop}
Let $\xx \subset \p^3(\fqc)$ be a nonstrange curve defined over $\fq$ with order sequence $(0,1,\varepsilon_2,\varepsilon_3)$ such that $\varepsilon_3=p^r$ for some $r>0$, with $p>3$. Let $1 \leq m_1<  m_{2}$, be a sequence of integers such that $m_2 \neq 2m_1$ and $\varepsilon_3<\min\{q^{m_1},q^{m_2-m_1},q^{|m_2-2m_1|}\}$. If the Gauss map of $\xx$ is purely inseparable and $\xx$ is $\F_{q^{m_i}}$-Frobenius nonclassical for $i=1,2$ such that the last term of the  $\F_{q^{m_i}}$-Frobenius order sequence equals to $\varepsilon_n$, then the order sequence of the strict dual $\xx^{'}$ of $\xx$ is $(0,1,q^{m_1}/\varepsilon_3,q^{m_2}/\varepsilon_3)$. In particular,  $\xx^{'}$ and its tangent surface $\Tan(\xx^{'})$ are nonreflexive. Furthermore, $\xx=\xx^{''}$ with $\gamma^{'} \circ \gamma=\Phi_{q^{m_2}}$ and the  $\F_{q^{m_{2}-m_1}}$-Frobenius order sequence of $\xx^{'}$ is $(0,1,q^{m_2}/\varepsilon_3)$.
\end{prop}
\begin{proof}
Following the proof of Theorem \ref{main0} and using the same notation, we arrive at $\rank(M_3^{h_k-h_1}(0,h_2-h_1))=2$. If $k=2$, we obtain $\varepsilon_3>q^{m_2-m_1}$ since $x_i$ is a separating element for some $i$, which gives a contradiction. If $k=1$, $\rank(M_3^{0}(0,h_2-h_1))=2$ means that $\Phi_{q^{m_2-m_1}}(P) \in L_1(P)$ for a generic $P \in \xx$, where $L_1(P)$ is the tangent line to $\xx$ at $P$. In particular, $\Phi_{q^{m_2-m_1}}(P) \in H_P(\xx)$, where $H_P(\xx)$ is the osculating hyperplane to $\xx$ at a generic $P \in \xx$. However, since the last term of the  $\F_{q^{m_i}}$-Frobenius order sequence equals to $\varepsilon_n$, we have that $\Phi_{q^{m_i}}(P) \in H_P(\xx)$ for $i=1,2$ for a generic $P$. Thus, denoting $m_3=m_2-m_1$ and $m_0=0$, we have
$$
\det\left(x_i^{q^{m_j}}\right)_{\substack{0 \leq i,j \leq 3 }} =0,
$$
which gives that $\varepsilon_3 \geq q^\ell$, for some $\ell \in \{q^{m_1},q^{m_2-m_1},q^{|m_2-2m_1|}\}$, and we obtain another contradiction. Therefore
$$
\sum\limits_{i=0}^{3}\left(D_{\zeta}^{(1)}x_i\right)^{p^{h_k-r}}z_i \neq 0
$$
for $k=1,2$, where $\zeta \in \fq(\xx)$ is a separating element. The result then follows by Remark \ref{obs1} and Theorem \ref{main-1}.
\end{proof}

%\subsection*{Acknowledgment}

%We would like to thank the anonymous referees for the many comments and suggestions that improved the presentation of this manuscript.

\vspace{0,5cm}\noindent {\em Author's address}:

\vspace{0.2 cm} \noindent Nazar ARAKELIAN \\
Centro de Matem\'atica, Computa\c c\~ao e Cogni\c c\~ao
\\ Universidade Federal do ABC \\ Avenida dos Estados, 5001 \\
CEP 09210-580, Santo Andr\'e SP
(Brazil).\\
 E--mail: {\tt n.arakelian@ufabc.edu.br}

\end{document}